\documentclass[11pt,a4paper,reqno]{amsart}
\usepackage[includehead,includefoot,margin=25mm]{geometry}
\usepackage{times}
\usepackage{amsfonts} %
\usepackage{amsmath} %
\usepackage{amssymb} %
\usepackage{amsthm} %
\usepackage{enumerate} %
\usepackage{bbm} %
\usepackage{graphicx} %
\usepackage{nicefrac} %
\usepackage{color} %
\usepackage{hyperref}
\hypersetup{
    colorlinks=true,       
    linkcolor=blue,          
    citecolor=blue,        
    filecolor=blue,      
    urlcolor=blue           
}

\newtheorem{theorem}{Theorem}[section] %
\newtheorem{corollary}[theorem]{Corollary} %
\newtheorem{lemma}[theorem]{Lemma} %
\newtheorem{proposition}[theorem]{Proposition} %
{\theoremstyle{remark} %
  \newtheorem{remark}[theorem]{Remark}} %
{\theoremstyle{definition} %
  \newtheorem{definition}[theorem]{Definition} %
  \newtheorem{example}[theorem]{Example} %
}

\newcommand{\PP}[0]{\ensuremath{\mathbb{P}}}
\newcommand{\CC}[0]{\ensuremath{\mathbb{C}}}
\newcommand{\ZZ}[0]{\ensuremath{\mathbb{Z}}}

\newcommand{\tvarphi}[0]{\ensuremath{\widetilde{\varphi}}}
\newcommand{\tpsi}[0]{\ensuremath{\widetilde{\psi}}}
\newcommand{\ttau}[0]{\ensuremath{\widetilde{\tau}}}
\newcommand{\tG}[0]{\ensuremath{\widetilde{G}}}

\newcommand{\Aut}[0]{\ensuremath{\operatorname{Aut}}}
\newcommand{\Lin}[0]{\ensuremath{\operatorname{Lin}}}
\newcommand{\diag}[0]{\ensuremath{\operatorname{diag}}}
\newcommand{\PGL}[0]{\ensuremath{\operatorname{PGL}}}
\newcommand{\GL}[0]{\ensuremath{\operatorname{GL}}}
\newcommand{\SL}[0]{\ensuremath{\operatorname{SL}}}
\newcommand{\Id}[0]{\ensuremath{\operatorname{Id}}}

\begin{document}

\title[On the liftability of the automorphism group of smooth hypersurfaces]{On the liftability of the automorphism group of smooth hypersurfaces of the projective space}

\author{V\'\i ctor Gonz\'alez-Aguilera}
\address{Departamento de Matem\'aticas, Universidad T\'ecnica
  Fe\-de\-ri\-co San\-ta Ma\-r\'\i a, Valpara\'\i
  so, Chile}  \email{victor.gonzalez@usm.cl}

\author{Alvaro Liendo} %
\address{Instituto de Matem\'atica y F\'isica, Universidad de Talca,
  Casilla 721, Talca, Chile} %
\email{aliendo@inst-mat.utalca.cl}

\author{Pedro Montero}
\address{Departamento de Matem\'aticas, Universidad T\'ecnica
  Fe\-de\-ri\-co San\-ta Ma\-r\'\i a, Valpara\'\i
  so, Chile}  \email{pedro.montero@usm.cl}

\date{\today}

\thanks{
  \mbox{\hspace{11pt}}{\it Key words}: Smooth hypersurfaces, $n$-folds, Automorphisms group of smooth hypersurfaces, $F$-liftability.\\
  \mbox{\hspace{11pt}} The second author was partially supported by Fondecyt Projects 1160864 and 1200502. The third author was partially supported by Fondecyt Projects 11190323 and 1200502.}

\begin{abstract}
  Let $X$ be a smooth hypersurface of dimension $n\geq 1$ and degree $d\geq 3$ in the projective space given as the zero set of a homogeneous form $F$. If $(n,d)\neq (1,3), (2,4)$ it is well known that every automorphism of $X$ extends to an automorphism of the projective space, i.e., $\operatorname{Aut}(X)\subseteq \operatorname{PGL}(n+2,\mathbb{C})$. We say that the automorphism group $\operatorname{Aut}(X)$ is $F$-liftable if there exists a subgroup of $\operatorname{GL}(n+2,\mathbb{C})$ projecting isomorphically onto $\operatorname{Aut}(X)$ and leaving $F$ invariant. Our main result in this paper shows that the automorphism group of every smooth hypersurface of dimension $n$ and degree $d$ is $F$-liftable if and only if $d$ and $n+2$ are relatively prime. We also provide an effective criterion to compute all the integers which are a power of a prime number and that appear as the order of an automorphism of a smooth hypersurface of dimension $n$ and degree $d$. As an application, we give a sufficient condition under which some Sylow $p$-subgroups of $\operatorname{Aut}(X)$ are trivial or cyclic of order $p$.
\end{abstract}

\maketitle

\section*{Introduction}

Let $n\geq 1$ and $d\geq 1$ be integers.  A hypersurface $X$ of dimension $n$ and degree $d$ of the projective space $\PP^{n+1}:=\PP^{n+1}(\CC)$ is the set of zeros of a homogeneous polynomial $F$ of degree $d$. A smooth hypersurface of dimension $n$ is also called an $n$-fold. Smooth hypersurfaces are classical objects in algebraic geometry since they are the simplest varieties one can define as they are given by only one equation. As such, they have been intensively studied and their geometry has shaped the development of classic and modern algebraic geometry. One noteworthy example is the central role that they play on the rationality problem since the very beginnings of algebraic geometry and specially after the seminal works of Iskovskikh and Manin \cite{IM71}, Artin and Mumford \cite{AM72}, and Clemens and Griffiths \cite{CG72}, concerning the L\"{u}roth problem for threefolds, and more recently after the new methods for studying stably rational varieties, introduced by Voisin in \cite{Voi15} and successfully developed by many other authors, for which we kindly refer the reader to the surveys \cite{Ko19} and \cite{Pey19}. Smooth hypersurfaces furnish a remarkably useful and concrete testing ground for many general theories.

If $d\geq 3$ and $(n,d)\neq (1,3), (2,4)$. A classical result due to Matsumura and Monsky \cite{MM64} states that every automorphism of a hypersurface $X$ is induced by an automorphism of the ambient projective space. Since the automorphism group of $\PP^{n+1}$ is the projective linear group $\PGL(n+2,\CC)$, we have then $\Aut(X)\subset \PGL(n+2,\CC)$. Given a subgroup $G$ of $\Aut(X)$, it is a natural question to ask if there exists a group $\tG\subset \GL(n+2,\CC)$ that projects isomorphically into $G$ via the the natural homomorphism $\pi\colon \GL(n+2,\CC)\rightarrow \PGL(n+2,\CC)$. If this is the case, we say that $G$ is liftable and we say that $\tG$ is a lifting of $G$. Let now $\tvarphi\in \tG$. Since $\pi(\tvarphi)$ is an automorphism of $X$, we have that $F$ is semi-invariant by $\tvarphi$, i.e., $\tvarphi^*(F)=\lambda F$. Furthermore, if we have that the homogeneous polynomial $F$ defining the hypersurface $F$ is invariant by every automorphism on $\tvarphi\in G$, i.e., $\tvarphi^*(F)=F$, then we say that $\tG$ is an $F$-lifting and we say that $G$ is $F$-liftable.

Our main result in this paper, contained in Theorem~\ref{all-liftable}, states that if $n\geq 1$, $d\geq 3$ and $(n,d)\neq (1,3), (2,4)$, then the automorphism group of every smooth hypersurface of dimension $n$ and degree $d$ in $\PP^{n+1}$ is $F$-liftable if and only if $d$ and $n+2$ are relatively prime. Our main inspirations for this work are the recent papers \cite{OY19} and \cite{WY19}, where the notion of $F$-liftability is introduced and applied to perform computer-aided calculations of all the maximal automorphism groups of smooth quintic threefolds and smooth cubic threefolds, respectively. Our results provide an alternative approach to some steps in their analysis, which we believe that can be useful to simplify that kind of computations for other smooth $n$-folds, by considering faithful representations of subgroups of $\Aut(X)$ (e.g. by applying Theorem \ref{all-liftable}) and the restrictions that we have for these liftings (e.g. those imposed by Lemma~\ref{determinant-3k}, Proposition \ref{most-p-SL} and Corollary \ref{sylow-small-order}, that we explain below).

In a previous work by two of the authors of this paper \cite{GL13} (see also \cite{GL11}), we gave a criterion for a power of a prime number $p^r$ to be the order of an automorphism of a smooth hypersurface of dimension $n$ and degree $d$ as long as $p$ does not divide $d$ nor $d-1$. In order to prove Theorem~\ref{all-liftable}, we improved our previous result to also include powers of prime numbers also dividing $d$ or $d-1$ (see in Theorem~\ref{power-prime-order}). Indeed, letting $n,d,p,r$ be positive integers with $d\geq 3$, $(n,d)\neq (1,3), (2,4)$ and $p$ prime, then $p^r$ is the order of an $F$-liftable automorphism of a smooth hypersurface of dimension $n$ and degree $d$ if and only if
  \begin{enumerate}[$(i)$]
  \item $p$ divides $d-1$ and $r\leq k(n+1)$, where $d-1=p^ke$ with $\gcd(p,e)=1$, or
    \item $p$ divides $d$ and and there exists $\ell\in\{1,2,\ldots,n+1\}$ such that $(1-d)^\ell\equiv 1 \mod p^r$, or
    \item $p$ does not divide $d(d-1)$ and there exists $\ell\in\{1,2,\ldots,n+2\}$ such that $(1-d)^\ell\equiv 1 \mod p^r$.
  \end{enumerate}

We apply Theorem~\ref{power-prime-order} as a tool to prove Lemma~\ref{determinant-3k}, which is a key ingredient to our main result in Theorem~\ref{all-liftable}. Lemma~\ref{determinant-3k} gives conditions on the values that the determinant $\det(\tvarphi)$ may take, where $\tvarphi$ is an $F$-lifting of an automorphism $\varphi$ of a smooth hypersurface of dimension $n$ and degree $d$ of order $p^r$ a power of a prime number. In particular, it states that whenever $p$ does not divide $d$ nor $d-1$, then $\det(\tvarphi)=1$. We think this lemma may be interesting on its own in other contexts.

As an application, we study the possible orders of certain Sylow $p$-subgroups of the automorphism group of smooth hypersurfaces in terms of some easy-to-compute numerical invariant. This allows us to provide a useful corollary stating that, under certain hypothesis that are usually fulfilled for smooth hypersurfaces of low dimension, the Sylow $p$-subgroups are of order at most $p$ (see Corollary~\ref{sylow-small-order}). In particular, for cubic hypersurfaces, we obtain in Example~\ref{orders-full-cubics} that the order of the automorphism group of every smooth cubic hypersurface of dimension at most five is not divisible by $p^2$ for every prime number different from $2$ and $3$.

\subsection*{Outline of the article} The content of the paper is organized as follows. In Section~\ref{sec:single-automorphism-lifting} we introduce the notions of liftability and $F$-liftability and we prove that if  $d$ and $n+2$ are relatively prime, then every cyclic subgroup of the automorphism group of a smooth hypersurface of dimension $n$ and degree $d$ is $F$-liftable. In Section~\ref{sec:orders-liftable} we prove Theorem~\ref{power-prime-order} and Lemma~\ref{determinant-3k} reported above. In Section~\ref{sec:liftability-aut} we prove our main result in Theorem~\ref{all-liftable}. Finally, in Section~\ref{sec:applications} we give the announced application to certain Sylow $p$-subgroups of the automorphism group of  a smooth hypersurface of dimension $n$ and degree $d$.

\subsection*{Note on a recent preprint} 

During the last stages of proof-reading of this article and a few days before our submission to the arXiv, the preprint \cite{Zhe20}, whose results are closely related and mostly complementary to ours, appeared on the arXiv. We decided not to change the text of our paper and rather describe the possible interactions here in the introduction. 

Our Theorem \ref{power-prime-order}, that we prove by adapting the proof of our previous result  \cite[Theorem~1.3]{GL13},
is also a corollary of \cite[Theorem 1.4]{Zhe20}. On the other hand, the $F$-liftability assumptions in \cite[Theorem 1.1]{Zhe20} and \cite[Theorem 1.3]{Zhe20} can be removed by our Theorem~\ref{all-liftable} in the rather general case where $\gcd(d,n+2)=\gcd(d,N)=1$. In the notation of \cite{Zhe20}, our $n+2$ corresponds to $N$, the dimension of the vector space defining the projective space $\PP^{n+1}=\PP^{N-1}$.

It is also worth stressing that all results in both papers are obtained by means of different approaches and the main objectives of both papers are complementary. We study optimal conditions to ensure $F$-liftability of the full automorphism group of a smooth hypersurface and our Theorem~\ref{power-prime-order} is a tool to this main result. On the other hand, the author of \cite{Zhe20} studies $F$-liftable abelian group actions on smooth hypersurfaces and derives  \cite[Theorem 1.4]{Zhe20} as a consequence.

\subsection*{Acknowledgements} This collaboration started during the conference ``Fourth Latin American School on Algebraic Geometry and its Applications (ELGA IV)" held at Universidad de Talca, where the authors benefited from fruitful discussions on the subject of this paper with Igor Dolgachev, to whom we are very grateful. We would like to thank the institution for the support and hospitality, and for making this event successful. The second author would also like to thank our friend and colleague Ana Cecilia de la Maza, a great part of the writing of this paper was done while he was in preemptive self-isolation at her cabin in the Chilean Andes.

\section{Liftability of automorphisms of smooth
  hypersurfaces}
  \label{sec:single-automorphism-lifting}

The aim of this section is to prove that every automorphism of every smooth hypersurface of dimension $n\geq 1$ and degree $d\geq 3$ with  $(n,d)\neq (1,3), (2,4)$ admits an $F$-lifting if and only $d$ and $n+2$ are relatively prime.

Fix for the whole paper a vector space $V$ over $\CC$ of dimension $n+2$, with $n\geq 1$. Let $\GL(V)$, $\SL(V)$ and $\PGL(V)$ be the general linear group, the special linear group and the projective linear group, respectively. We denote by $\pi\colon \GL(V)\rightarrow \PGL(V)$ the canonical projection. For an automorphism $\tvarphi\colon V\rightarrow V$ in $\GL(V)$ we denote its image $\pi(\tvarphi)$ by $\varphi$. The automorphism $\tvarphi$ induces an automorphism $\tvarphi^*\colon V^*\rightarrow V^*$ on the dual space $V^*$ given by $\tvarphi^*(L)=L\circ \tvarphi$. Consequently, it also induces a graded automorphism $\tvarphi^*\colon S(V^*)\rightarrow S(V^*)$ on the symmetric algebra $S(V^*)$ of the vector space $V^*$ given also by $\tvarphi^*(F)= F\circ\tvarphi$. This automorphism restricts to an automorphism $\tvarphi^*\colon S^d(V^*)\rightarrow S^d(V^*)$ of forms of degree $d$.

An automorphism of an algebraic variety $X$ is a regular map $X\rightarrow X$ having a regular inverse map. The group of all automorphisms of $X$ is denoted by $\Aut(X)$. Let $\PP(V)$ be the complex projective space of dimension $n+1$ of the vector space $V$. The automorphism group of $\PP(V)$ is the projective linear group $\PGL(V)$. Let now $X$ be a hypersurface of $\PP(V)$ given as the zero set of a homogeneous form $F\in S^d(V^*)$ of degree $d$, the group of linear automorphisms, denoted by $\Lin(X)$, is the subgroup of $\Aut(X)$ of automorphisms that extend to an automorphism of the ambient space $\PP(V)$, i.e., $\Lin(X)=\{\varphi\in \PGL(V)\mid \varphi(X)=X\}$.

In full generality, the group $\Lin(X)$ is a proper subgroup $\Aut(X)$, nevertheless for smooth hypersurface, we have the following classical theorem \cite{MM64}, see also \cite[\S 6]{Ko19} and the references therein.

\begin{theorem} \label{matsumura-monsk} %
  Let $X$ and $Y$ be smooth hypersurfaces of dimension $n\geq 1$ and degree $d\geq 3$ in the complex projective space $\PP(V)$ and let $\tau\colon X\rightarrow Y$ is an isomorphism. If $(n,d)\neq  (1,3), (2,4)$, then $\tau$ is the restriction of a linear automorphism $\PP(V)\rightarrow \PP(V)$ in $\PGL(V)$. In particular, every automorphism of $X$ is linear.  Moreover, $\Aut(X)$ is a finite group.
\end{theorem}

\begin{remark}
  In the sequel, we will only consider smooth hypersurfaces of the projective space of dimension $n$ and degree $d$ with $(n,d)\neq (1,3), (2,4)$. We will make use of the fact that  $\Aut(X)=\Lin(X)$ and it is finite without referring to Theorem~\ref{matsumura-monsk}.
\end{remark}

Let us now introduce the notion of liftability that we will use in this paper. It was first considered by Oguiso and Yu in \cite[\S 4]{OY19}.

\begin{definition}\label{definition:liftability} \
  \begin{enumerate}[(1)]
  \item Let $G$ be a subgroup of $\PGL(V)$. We say that a subgroup $\tG \subset \GL(V)$ is a lifting of $G$ if the restriction of the canonical projection $\pi:\GL(V)\to \PGL(V)$ to $\tG$ induces an isomorphism $\pi|_{\widetilde{G}}:\tG\xrightarrow{\sim} G$. In this case, we say that $G$ is liftable. Furthermore, we say that $\varphi\in\PGL(V)$ is liftable if the generated group $\langle \varphi\rangle$ is liftable and we denote by $\tvarphi$ a lifting of $\varphi$, i.e., $\tvarphi$ is an element of $\GL(V)$ of the same order as $\varphi$ and such that $\pi(\tvarphi)=\varphi$

  \item Let now $X\subset \PP(V)$ be a smooth hypersurface of dimension $n\geq 1$ and degree $d\geq 3$ given by the homogeneous form $F\in S^d(V^*)$ and let $G$ be a subgroup of $\Aut(X)$. Assume also that $(n,d)\neq (1,3), (2,4)$. We say that $\tG \subset \GL(V)$ is an $F$-lifting of $G$ if $\tG$ is a lifting of $G$ and $\tvarphi^*(F)=F$, for all $\tvarphi\in \tG$. In this case, we say that $G$ is $F$-liftable. Furthermore, we say that $\varphi\in\Aut(X)$ is $F$-liftable if the generated group $\langle \varphi\rangle$ is $F$-liftable and we denote by $\tvarphi$ a lifting of $\varphi$, i.e., $\tvarphi$ is an element of $\GL(V)$ of the same order as $\varphi$ such that $\pi(\tvarphi)=\varphi$ and $\tvarphi^*(F)=F$.
  \end{enumerate}
\end{definition}
 
A priori, the notion of $F$-liftability defined by Oguiso and Yu depends on the particular embedding of $X$ into $\PP(V)$. In the following lemma, we show that indeed $F$-liftability is independent of the embedding.

\begin{proposition} 
  Let $X$ and $X'$ be smooth hypersurfaces of dimension $n\geq 1$ and degree $d\geq 3$ given by and forms $F$ and $F'$ in $S^d(V^*)$, respectively. Let also $\tau\colon X\rightarrow X'$ be an isomorphism and assume $(n,d)\neq (1,3), (2,4)$.  If $G\subseteq \Aut(X)$ is a subgroup, then $G$ is $F$-liftable if and only if $\tau\circ G\circ\tau^{-1}\subseteq \Aut(X')$ is $F'$-liftable.
\end{proposition}

 \begin{proof}
   By Theorem~\ref{matsumura-monsk} we have that $\tau$ is the restriction of an automorphism $\tau\colon \PP(V)\rightarrow \PP(V)$ in $\PGL(V)$. Let $\ttau\in\GL(V)$ be such that $\pi(\ttau)=\tau$. Since  $X$ is given by the form $F$, then $X'$ is given by the  form $(\ttau^*)^{-1}(F)=F\circ\ttau^{-1}$. Hence, up to replacing $\ttau$ by $\lambda \ttau$ with $\lambda\in \CC^\times$, we can assume $F'=F\circ\ttau^{-1}$. Let now $\varphi\in G\subseteq \Aut(X)$ be a linear automorphism of $X$ and let $\tvarphi\in\GL(V)$ be such that $\pi(\tvarphi)=\varphi$. Clearly, $\tau\circ\varphi\circ\tau^{-1}$ is an automorphism of $Y$ and $\pi(\ttau\circ\tvarphi\circ\ttau^{-1})=\tau\circ\varphi\circ\tau^{-1}$. Moreover, 
   $$\tvarphi^*(F)=F \Leftrightarrow F\circ\tvarphi=F\Leftrightarrow  F'\circ\ttau\circ\tvarphi=F'\circ\ttau \Leftrightarrow F'\circ\ttau\circ\tvarphi\circ\ttau^{-1}=F' \Leftrightarrow \left(\ttau\circ\tvarphi\circ\ttau^{-1}\right)^*(F')=F'$$
   This yields that $\tvarphi$ is a lifting of $\varphi$ if and only if $\ttau\circ\tvarphi\circ\ttau^{-1}$ is a lifting of $\tau\circ\varphi\circ\tau^{-1}$. Finally, it is clear that the canonical map $\tG \to G$ is an isomorphism if and only if $\ttau\circ\tG\circ\ttau^{-1} \to \tau\circ G\circ \tau^{-1}$ is an isomorphism. This concludes the proof.
 \end{proof}

 We now fix some notation for computations and examples that need to be performed in coordinates. Let $\beta=\{\beta_0,\ldots,\beta_{n+1}\}$ be a basis of $V$ and let $\beta^*=\{x_0,\ldots,x_{n+1}\}$ be the corresponding dual basis of $V^*$. This choice induces a canonical isomorphism of the symmetric algebra $S(V^*)$ and the polynomial ring $\CC[x_0,\ldots,x_{n+1}]$. Under this isomorphism, a form $F \in S^d(V^*)$ of degree $d$ corresponds to a homogeneous polynomial of total degree $d$. As usual, the degree of $F$ with respect to variable $x_i$, denoted by $\deg_{x_i}(F)$, corresponds to the degree of $F$ seen as a polynomial in $A[x_i]$ where $A$ is the polynomial ring in the variables $\{x_0,\ldots,\widehat{x_i},\ldots,x_{n+1}\}$. 

 \begin{lemma} \label{diagonalization} %
   Let $\varphi\in \PGL(V)$ be an automorphism of $\PP(V)$ of finite order $q$.  Then, there exists a lifting $\tvarphi\in\GL(V)$ and every other lifting of $\varphi$ is given by $\lambda\tvarphi$ with $\lambda$ a $q$-root of unity. Moreover, $\tvarphi$ and $\tvarphi^*$ are diagonalizable and all their eigenvalues are $q$-roots of unity. Finally, up to a change of basis, we can assume that $\tvarphi$ and $\tvarphi^*$ are given by the same diagonal matrix $\diag(\xi^{\sigma_0},\ldots,\xi^{\sigma_{n+1}})$ in the dual bases $\beta$ and $\beta^*$, respectively, where $\xi$ is a primitive $q$-root of unity and $\sigma_i$ are integers.
 \end{lemma}

 \begin{proof}
   Let $\tvarphi\in\GL(V)$ be any automorphism such that $\pi(\tvarphi)=\varphi$.  Since $\varphi$ has order $q$ we have that $\tvarphi^q=\lambda\Id$. Replacing $\tvarphi$ by $\tfrac{1}{\lambda'}\tvarphi$, where $(\lambda')^q=\lambda$ we obtain that $\tvarphi$ is also of order $q$ and so it is a lifting of $\varphi$.  Furthermore, if $\tvarphi'$ is another lifting of $\varphi$, then $\tvarphi'=\lambda\tvarphi$ and since both have order $q$, we obtain that $\lambda$ is a $q$-root of unity.  Now, since $\tvarphi^q=\Id$, the minimal polynomial of $\tvarphi$ divides $X^q-1$ which is a polynomial with only simple roots. This yields that $\tvarphi$ is diagonalizable and all its eigenvalues are $q$-roots of unity. This shows that $\tvarphi$ has the desired form. Finally, in the dual basis $\beta^*$, the dual  $\tvarphi^*$ is given by the transpose matrix of $\tvarphi$ which yields that $\tvarphi$ and $\tvarphi^*$ admit the same diagonalization in their respective bases.
 \end{proof}

 Let us now prove the following lemma stating that when a form $F\in S^d(V^*)$ of degree $d$ has a particular shape, then the corresponding hypersurface $X$ in $\PP(V)$ is singular.
 
\begin{lemma} \label{simple}
  Let $X$ be a hypersurface of dimension $n$ and degree $d$ in $\PP(V)$ given by the homogeneous form $F\in S^d(V^*)$. Assume there is a decomposition $V=V_0\oplus V_1\oplus V_2$ and that
  $$F\in S^{d-1}(V_0^*)\otimes S^1(V_1^*)\oplus S^1(V_0^*)\otimes S^{d-1}(V_1^*\oplus V_2^*)\oplus S^d(V_1^*\oplus V_2^*)\,.$$ %
  If $\dim V_0>\dim V_1$ then $X$ is singular.
\end{lemma}

\begin{proof}
  The projective space $\PP(V_0)$ is a subvariety of $\PP(V)$ isomorphic to a projective space of dimension $\dim V_0-1$. The points of $\PP(V_0)$ are given by the points $[v_0\oplus 0\oplus 0]$ inside $\PP(V)$, with $v_0\in V_0$. By the shape of $F$, the subvariety $\PP(V_0)$ is also contained in $X$. If we take bases of $V_0$, $V_1$ and $V_2$, the Jacobian Criterion applied to $F$ with respect to an element in the dual basis of $V_0^*$ gives an equation that is trivially satisfied for points in $\PP(V_0)$. The same holds for the Jacobian Criterion applied to $F$ with respect to an element in the dual basis of $V_2^*$. Now, the Jacobian Criterion applied to $F$ with respect to element in the dual basis of $V_1^*$ give exactly $\dim V_1$ equations, but by Euler's Lemma \cite[Chapter~I, Exercise~5.8]{Har77}, at most $\dim V_1-1$ of these equations are independent. Hence, the Jacobian Criterion for points in the projective space $\PP(V_0)$ that has dimension $\dim V_0-1$ gives $\dim V_1-1$ equations. If $\dim V_0>\dim V_1$, B\'ezout Theorem \cite[Chapter~I, Theorem~7.2]{Har77} shows that there is at least one solution and so $X$ is singular.
\end{proof}

\begin{remark} \label{simple-original}
  The case where $\dim V_0=1$ and $\dim V_1=0$ of this lemma has been used previously in several instances. Is has usually been stated in coordinates, if we take any non-trivial element on $V^*_0$ to be the $i$-th element $x_i$ in the basis of $V^*$, then Lemma~\ref{simple} states that if $\deg_{x_i}(F)<d-1$, then $X$ is singular. See \cite[Lemma~2.3]{GL11}, \cite[Lemma~1.2]{GL13}, \cite[Proposition~1.3~(1)]{OY19}, \cite[Proposition~4.3~(1)]{WY19}.
\end{remark}

As a first step into proving our main result, in the following two lemmas we state some straightforward consequences of the definition of  $F$-liftability for automorphisms of order $q$ of a smooth hypersurfaces of degree $d$.

\begin{lemma} \label{basic-liftings} %
  Let $X$ be a hypersurface of dimension $n\geq 1$ and degree $d\geq 3$ in $\PP(V)$ given by the homogeneous form $F\in S^d(V^*)$, with $(n,d)\neq (1,3), (2,4)$. Let $\varphi$ be an automorphism of $X$ of order $q$ and let $\tvarphi$ and $\tvarphi'$ be automorphisms in to $\GL(V)$. If $\tvarphi$ is an $F$-lifting, then $\tvarphi'$ is an $F$-lifting if and only if $\tvarphi'=\lambda \tvarphi$ with $\lambda$ a $\gcd(d,q)$-root of unity.

  In particular, the following hold.
  \begin{enumerate}[(i)]
  \item If $\gcd(d,q)=1$, then if an $F$-lifting exists, it is unique.
  \item If $q$ divides $d$ we have the following alternative: either $\varphi$ is not $F$-liftable, or every lifting of $\varphi$ is an $F$-lifting.
  \end{enumerate}
\end{lemma}

\begin{proof}
  Assume that $\tvarphi\in \GL(V)$ is an $F$-lifting of $\varphi$. By Lemma~\ref{diagonalization}, $\tvarphi'$ is a lifting of $\varphi$ if and only if $\tvarphi'=\lambda\tvarphi$, with $\lambda$ a $q$-root of unity. Furthermore, $\tvarphi'(F)=(\lambda\tvarphi)^*(F)=\lambda^d\tvarphi^*(F)$. Hence, we conclude that $\tvarphi'$ is an $F$-lifting if and only if $\lambda$ is a $q$-root of unity and $d$-root of unity. Now, $(i)$ follows directly and $(ii)$ follows from Lemma~\ref{diagonalization}.
\end{proof}

\begin{lemma} \label{order-3} %
  Let $X$ be a smooth hypersurface of dimension $n\geq 1$ and degree $d\geq 3$ in $\PP(V)$ given by the homogeneous form $F\in S^d(V^*)$, with $(n,d)\neq (1,3), (2,4)$. Let $p$ be a prime number dividing $d$ and let $\varphi$ be an automorphism of $X$ of order $p$. If $\varphi$ is not $F$-liftable, then $p$ also divides $n+2$.
\end{lemma}

\begin{proof}
  Let $p$ be a prime number dividing $d$ and let $\varphi$ be an automorphism of $X$ of order $p$ that is not $F$-liftable. Let also $\tvarphi$ be a lifting of $\varphi$ to $\GL(V)$. By Lemma~\ref{diagonalization}, we have that $\tvarphi$ is diagonalizable and its eigenvalues are all $p$-roots of unity. Hence, we have $\tvarphi^*(F)=\xi F$ with $\xi$ a $p$-root of unity.  Since $\varphi$ is not $F$-liftable, $\xi\neq 1$ and since $p$ is prime, $\xi$ is a primitive $p$-root of unity. For every integer $i$, we let $V(i)$ be the eigenspace of $\tvarphi\colon V\to V$ associated to the eigenvalue $\xi^i$. We have $V(i)=V(i+kp)$ for every integer $k$. Fix an integer $j$, and let $V_0=V(j)$, $V_1=V(j+1)$ and $V_2$ a complementary of $V_0\oplus V_1$.  Since $\tvarphi^*(F)=\xi F$ and $j(d-1)+(j+1)=jd+1\equiv 1 \mod p$, we have that
  $$F\in S^{d-1}(V_0^*)\otimes S^1(V_1^*)\oplus S^1(V_0^*)\otimes S^{d-1}(V_1^*\oplus V_2^*)\oplus S^d(V_1^*\oplus V_2^*)\,.$$
  Since $X$ is smooth, applying Lemma~\ref{simple} yields $\dim V_0=\dim V(j)\leq \dim V_1=\dim V(j+1)$ for every integer $j$. Hence,
  $$\dim V(0)\leq \dim V(1)\leq\cdots\leq \dim V(p)=V(0)$$
  This yields $\dim V(i)=\dim V(0)$ for every integer $i$. Hence, we obtain $n+2=\dim V=p\dim V(0)$ proving the lemma.
\end{proof}

We recollect now several examples that will be required in the sequel. Our first example provides a converse to Lemma~\ref{order-3}. Indeed, for every prime number $p$ that divides $d$ and $n+2$, we provide a smooth hypersurface $X$ of degree $d$ and dimension $n$ and an automorphism $\varphi$ of order $p$ of $X$ that is not $F$-liftable.

\begin{example} \label{Klein-non-liftable} %
  The Klein hypersurface $X_K$ of dimension $n$ and degree $d$ is the hypersurface in $\PP(V)$ given by the homogeneous form
  $$K=x_0^{d-1}x_1+x_1^{d-1}x_2+\ldots+x_n^{d-1}x_{n+1}+x_{n+1}^{d-1}x_0\,.$$
  It is well known that $X_K$ is smooth \cite[Example~3.5]{GL13}. Let $p$ be a prime number dividing $\gcd(d,n+2)$ and assume that $n+2=pm$. Let $\tvarphi$ be the automorphism of order $p$ in $\GL(V)$ given in the basis $\beta$ by the diagonal matrix
$$\tvarphi=\diag(\underbrace{1,\xi,\xi^2,\ldots,\xi^{p-1},\ldots\ldots,1,\xi,\xi^2,\ldots,\xi^{p-1}}_{\mbox{\tiny{$m$-times}}}),$$
where $\xi$ is a primitive $p$-root of unity. A straightforward computation shows that $\tvarphi$ induces an automorphism $\varphi$ of the Klein hypersurface and $\tvarphi^*(K)=\xi K$. Hence, the automorphism $\varphi$ is not $K$-liftable by Lemma~\ref{basic-liftings} $(ii)$. In particular, $\Aut(X_K)$ is not $K$-liftable.
\end{example}

\begin{example} \label{Fermat} %
  The Fermat hypersurface $X_F$ of dimension $n$ and degree $d$ is the hypersurface in $\PP(V)$ given by the homogeneous form
  $$F=x_0^d+x_1^d+\ldots+x_n^d+x_{n+1}^d\,.$$
  A direct computation shows that $X_F$ is smooth. Furthermore, the automorphism group of $X_F$ is
  $$\Aut(X_F)=S_{n+2}\ltimes (\ZZ/d\ZZ)^{n+1}\,,$$
  where $S_{n+2}$ acts by permutation of variables and $(\ZZ/d\ZZ)^{n+1}$ acts diagonally in the first $n+1$ variables, see for instance \cite{Kon02}. In this case, $\Aut(X_F)$ is $F$-liftable.
\end{example}

\begin{example} \label{sum-forms} %
  Let $F\in S^d(V^*)$ and $F'\in S^d(W^*)$ be smooth homogeneous forms such that the corresponding hypersurfaces in $\PP(V)$ and $\PP(W)$ are smooth. Then the hypersurface given by $F+F'\in S^d(V^*\oplus W^*)$ in $\PP(V\oplus W)$ is also smooth.  This follows directly from the Jacobian Criterion since in this case, once we fix coordinates, the equations given by the vanishing of the partial derivatives split into two disjoint sets of variables. Thus, the existence of a singular point on the hypersurface defined by $F+F'$ in $\PP(V\oplus W)$ would imply the same for the hypersurface defined by $F$ or $F'$.
\end{example}

\begin{example} \label{strange-cover} %
  Let $X$ be the hypersurface of dimension $n$ and degree $d$ in $\PP(V)$ given by the homogeneous form
  $$F=
  \sum_{i=0}^{r-1} x_{i}^{d-1}x_{i+1} +\sum_{i=r}^{n+1}x_i^d\quad\mbox{where}\quad 0<r<n+2\,.$$
It is worth mentioning that $F$ is not the sum of two smooth homogeneous forms as in Example \ref{sum-forms}. However, a straightforward computation shows that the hypersurface $X \in \PP(V)$ corresponding to the homogeneous form $F$ is smooth. 
\end{example}

Our main result in this section is Proposition~\ref{proposition:char of liftability} showing that every automorphism of every smooth hypersurfaces of dimension $n$ and degree $d$ admit an $F$-lifting if and only if $d$ and $n+2$ are relatively prime. We first show the following lemma for particular automorphisms.

\begin{lemma} \label{non-liftable-single}
Let $X$ be a smooth hypersurface of dimension $n\geq 1$ and degree $d\geq 3$ in $\PP(V)$ given by the homogeneous form $F\in S^d(V^*)$, with $(n,d)\neq (1,3), (2,4)$. Let $\varphi$ be an automorphism of $X$ of order $q$ and let $\tvarphi$ be a lifting of $\varphi$ to $\GL(V)$ with $\tvarphi^*(F)=\xi^c F$, where $\xi$ is a primitive $q$-root of unity and $c\in \ZZ$. Then $\varphi$ is $F$-liftable if and only if $\gcd(d,q)$ divides $c$. In particular, if $\varphi$ is not $F$-liftable, then $d$ and $n+2$ are not relatively prime.
\end{lemma} 

\begin{proof}
  By Theorem~\ref{matsumura-monsk} the order of $\varphi$ is finite, say $q\geq 1$. By Lemma~\ref{diagonalization}, there exists a lifting $\tvarphi$ of $\varphi$ to $\GL(V)$ and all the eigenvalues of $\tvarphi^*$ are $q$-roots of unity. Hence, we have that $\tvarphi^*(F)=\xi^c F$, for some integer $c$, where $\xi$ is a primitive $q$-root of unity.
  
   Assume first that $\gcd(d,q)$ divides $c$. By B\'ezout identity, there exists $a$ and $b$ be integers such that $qa-db=c$. A straightforward computation shows that $\xi^b\tvarphi$ is an $F$-lifting of $\varphi$. Hence $\varphi$ is $F$-liftable in this case.

  Assume now that $\gcd(d,q)$ does not divide $c$. Let $p$ be a prime factor of $\gcd(d,q)$ not dividing $c$ and write $q=pr$. Letting $\psi=\varphi^{r}$, which is an automorphism of $X$ of order $p$, we set $\tpsi=\tvarphi^r$ in $\GL(V)$. We have $\tpsi^*(F)=\xi^{rc}F=\omega^cF$ where $\omega=\xi^r$ is a primitive $p$-root of unity and $\omega^c\neq 1$ since $p$ does not divide $c$. By Lemma~\ref{basic-liftings}~$(ii)$ it follows that $\psi$ is not $F$-liftable. This implies that $\varphi=\psi^r$ is also not $F$-liftable.  
  
  The last statement follows by Lemma~\ref{order-3} applied to $\psi$ that shows that $p$ also divides $n+2$.
\end{proof}

\begin{proposition}\label{proposition:char of liftability}
  Let $n\geq 1$ and $d\geq 3$ with $(n,d)\neq (1,3), (2,4)$. Then every automorphism of every smooth hypersurface of dimension $n$ and degree $d$ is $F$-liftable if and only if $d$ and $n+2$ are relatively prime.
\end{proposition}

\begin{proof}
  The ``only if'' part follow directly from Example~\ref{Klein-non-liftable} that provides for every prime number $p$ dividing $\gcd(d,n+2)$, an automorphism of the Klein hypersurface of dimension $n$ and degree $d$ of that is not $F$-liftable. The ``if'' part, follows from the last statement in Lemma~\ref{non-liftable-single}.
\end{proof}

\begin{remark}
In the particular case of smooth cubic threefolds, Proposition~\ref{proposition:char of liftability} is implied by  \cite[Proposition~4.7 and Lemma~4.9]{WY19}. Moreover, this result also agrees with \cite[Example~4.7]{OY19}, where the authors prove that some quintic threefolds have automorphism that are not $F$-liftable. In this last case  $d=n+2=5$. Both these results were the main inspiration for this paper. We generalize \cite[Proposition~4.7]{WY19} in Lemma~\ref{lemma:other-p} below.
\end{remark}

\section{Liftable automorphisms of order a power of a prime}
\label{sec:orders-liftable}

In \cite{GL11} we gave an effective criterion to find all the prime orders of automorphisms of smooth cubic hypersurfaces and in \cite[Theorem~1.3]{GL13} we extended this criterion to hypersurfaces of any degree $d$ and automorphism of order $p^r$ relatively prime to $d$ and $d-1$. For our main result, we need information about  $F$-liftable automorphisms of orders $p^r$ for every prime number dividing $d$, see Lemma~\ref{determinant-3k}. For this reason, we begin this section by extending in Theorem~\ref{power-prime-order} the criterion in \cite[Theorem~1.3]{GL13} to include the case of $F$-liftable automorphisms of order $p^r$ a power of a prime without restrictions on divisibility. 

Recall that by Lemma~\ref{diagonalization}, every automorphism $\tvarphi$ in $\GL(V)$ of finite order $q$ can be presented, in an appropriate basis, by a diagonal matrix $\diag(\xi^{\sigma_0},\ldots,\xi^{\sigma_{n+1}})$, where $\xi$ is a primitive $q$-root of unity and $\sigma_i$ are integers. Since $\xi^a=\xi^{a+q}$ the integers $\sigma_i$ are naturally defined only up to a multiple of $q$. Hence, the automorphism $\tvarphi$ is uniquely determined by it order $q$ and the vector $\sigma=(\sigma_0,\sigma_1,\ldots,\sigma_{n+1})\in (\ZZ/q\ZZ)^{n+2}$, where $\ZZ/q\ZZ$ is the group of integers modulo $q$. We call $\sigma$ the signature of  $\varphi$.

\begin{theorem} \label{power-prime-order}
  Let $n,d,p,r$ be positive integers with $d\geq 3$, $(n,d)\neq (1,3), (2,4)$ and $p$ prime.  Then $p^r$ is the order of an $F$-liftable automorphism of a smooth hypersurface of dimension $n$ and degree $d$ if and only if
  \begin{enumerate}[$(i)$]
  \item $p$ divides $d-1$ and $r\leq k(n+1)$, where $d-1=p^ke$ with $\gcd(p,e)=1$, or
    \item $p$ divides $d$ and and there exists $\ell\in\{1,2,\ldots,n+1\}$ such that $(1-d)^\ell\equiv 1 \mod p^r$, or
    \item $p$ does not divide $d(d-1)$ and there exists $\ell\in\{1,2,\ldots,n+2\}$ such that $(1-d)^\ell\equiv 1 \mod p^r$.

  \end{enumerate}
\end{theorem}
  
\begin{proof}
  Let $X$ be a smooth hypersurface of dimension $n$ and degree $d$ in $\PP(V)$ given by a homogeneous form $F\in S^d(V^*)$. Assume that $X$ admits an $F$-liftable automorphism $\varphi$ of order $p^r$ and let $\tvarphi$ be an $F$-lifting. Up to a change of basis, we may and will assume that $\tvarphi$ is diagonal with signature $\sigma=(\sigma_0,\ldots,\sigma_{n+1})\in(\ZZ/p^r\ZZ)^{n+2}$.

  To prove the ``only if'' part of the theorem, let $k_0\in \{0,\ldots,n+1\}$ be such that $\sigma_{k_0}$ is relatively prime to $p$. This is always possible since otherwise, the order of $\varphi$ would have order smaller that $p^r$. By Remark~\ref{simple-original}, the homogeneous form $F$ contains a monomial $x_{k_0}^{d-1}x_{k_1}$ as a summand with non-zero coefficient, for some $k_1\in\{0,\ldots,n+1\}$, where it is possible that $k_1=k_0$. Since $\tvarphi$ is diagonal and $\tvarphi^*(F)=F$ we have that $x_{k_0}^{d-1}x_{k_1}$ is an eigenvector of $\tvarphi^*$ associated to the eigenvalue $1$. Hence $(d-1)\sigma_{k_0}+\sigma_{k_1}\equiv 0 \mod p^r$, and so
  \begin{align}
    \label{equiva}  
    \sigma_{k_1}\equiv(1-d)\sigma_{k_0}\quad \mod p^r\,.
  \end{align}

  Applying the above argument with $k_0$ replaced by $k_1$, we obtain that the homogeneous form $F$ contains a monomial $x_{k_1}^{d-1}x_{k_2}$ as a summand with non-zero coefficient, for some $k_2\in\{0,\ldots,n+1\}$ and $\sigma_{k_2}\equiv(1-d)\sigma_{k_1}\quad \mod p^r$. Iterating this process, for all $i\in\{3,\ldots,n+2\}$ we let $k_i\in\{0,\ldots,n+1\}$ be such that the homogeneous form $F$ contains a monomial $x_{k_1}^{d-1}x_{k_2}$ as a summand with non-zero coefficient.  By \eqref{equiva}, we have
  \begin{align*}
    \sigma_{k_i}\equiv (1-d)\sigma_{k_{i-1}}\equiv (1-d)^2\sigma_{k_{i-2}}
    \equiv(1-d)^i\sigma_{k_0} \mod p^r ,\ \forall i\in\{2,\ldots,n+2\}\,.
  \end{align*}

  Since $k_i\in\{0,\ldots,n+1\}$ there are at least two $i,j\in\{0,\ldots,n+2\}$, $i>j$ such that $k_i=k_j$. Thus $\sigma_{k_i}=\sigma_{k_j}$, and since $\sigma_{k_i}\equiv (1-d)^i\sigma_{k_0}\mod p^r$ and  $\sigma_{k_j}\equiv (1-d)^j\sigma_{k_0} \mod p^r$ with $\sigma_{k_0}$ relatively prime to $p^r$ we have
  \begin{align}
    \label{equiva2}
  (1-d)^{i}\equiv (1-d)^j \quad \mod p^r\,,
  \end{align}

  Assume first that $p$ does not divide $d(d-1)$. Then \eqref{equiva2} is equivalent to $(1-d)^{i-j}\equiv 1 \mod p^r$. Setting $\ell=i-j$ yields the ``only if'' part in this case.

  Assume now that $p$ divide $d$. Then \eqref{equiva2} is again equivalent to $(1-d)^{i-j}\equiv 1 \mod p^r$. Setting $\ell=i-j$ yields that there exists $\ell\in\{1,2,\ldots,n+2\}$ such that $(1-d)^\ell\equiv 1 \mod p^r$. Assume the smallest $\ell$ satisfying this property is $n+2$. Then $j=0$, $i=n+2$ and, up to reordering of the basis, the signature of $\tvarphi$ is given by  $$\sigma=\big(\sigma_{k_0},(1-d)\sigma_{k_0},(1-d)^2\sigma_{k_0},\ldots,(1-d)^{n+1}\sigma_{k_0}\big)\mod p^r\,.$$
Since $p$ divides $d$, we obtain that the signature of $\tvarphi^{p^{r-1}}$ is 
\begin{align*}
p^{r-1}\cdot\sigma&=\big(p^{r-1}\sigma_{k_0},p^{r-1}(1-d)\sigma_{k_0},p^{r-1}(1-d)^2\sigma_{k_0},\ldots,p^{r-1}(1-d)^{n+1}\sigma_{k_0}\big) \mod p^r \\
  &=p^{r-1}\sigma_{k_0}\cdot (1,1,\ldots,1) \mod p^r\,.
\end{align*}
Hence, we obtain that in this particular case $\tvarphi$ and so $\varphi$ is of order $p^{r-1}$ instead of order $p^r$ and so $\ell<n+2$ this yields the ``only if'' part in this case.

  Finally, assume that $p$ divides $d-1$ and let $d-1=p^ke$ with $\gcd(p,e)=1$. If $kj<r$, then \eqref{equiva2} is equivalent to $(1-d)^{i-j}\equiv 1\mod p^{r-kj}$ which is impossible since $1$ is not divisible by $p$ but $(1-d)^{i-j}$ and $p^{r-kj}$ are divisible by $p$. Hence, we conclude $kj\geq r$. In this case  both sides of \eqref{equiva2} are divisible by $p^r$ and so the equation always holds. We conclude $r\leq kj$ and the maximal value that $j$ can attain is $n+1$ since $j<i$ and both are in the set $\{0,\ldots,n+2\}$ . Hence, we conclude $r\leq k(n+1)$. This yields the ``only if'' part in this last case.

\medskip

To prove the ``if'' part, assume first that $p$ divides $d-1$ and let $d-1=p^ke$ with $\gcd(p,e)=1$. To prove the theorem in this case, it is enough to provide a smooth hypersurface of dimension $n$ and degree $d$ admitting an automorphism of order $p^r$ where $r=k(n+1)$. We let $F\in S^d(V^*)$ be the form
  $$F=
  \sum_{i=0}^{n} x_{i}^{d-1}x_{i+1} +x_{n+1}^d\,.$$ %
  By Example~\ref{strange-cover}, the hypersurface $X \in \PP(V)$ corresponding to the homogeneous form $F$ is smooth. Furthermore, a straightforward computation shows that $F$ is invariant by the diagonal automorphism $\tvarphi$ of order $p^r$ in $\GL(V)$ with signature
    $$\sigma=\big(1,(1-d),(1-d)^{2},\ldots,(1-d)^{n},0  \big)\,.$$ %
    The fact that the signature contains coordinates with $0$ and $1$ ensures that the order of $\tvarphi$ is $p^r$. Indeed, assume that $\tvarphi^{p^k}=\lambda\Id$ for some $k<r$. Then taking the coordinates with $0$ and $1$ in the signature yields $0\equiv p^k\mod p^r$ which is a contradiction. Hence, $\tvarphi$ induces an $F$-liftable $\varphi$ of $X$ of order $p^r$, proving the theorem in this case.

    We deal with the last two cases simultaneously. If $p$ does not divide $d(d-1)$ we let $\ell\in\{1,2,\ldots,n+2\}$ such that $(1-d)^\ell\equiv 1 \mod p^r$. If $p$ divides $d$ we let $\ell\in\{1,2,\ldots,n+1\}$ such that $(1-d)^\ell\equiv 1 \mod p^r$. In both cases, we let $F\in S^d(V^*)$ be the form
  $$F=
  \sum_{i=0}^{\ell-2} x_{i}^{d-1}x_{i+1}+x_{\ell-1}^{d-1}x_0 +\sum_{i=\ell}^{n+1}x_i^d\,.$$ By Examples~\ref{Klein-non-liftable}, \ref{Fermat} and \ref{sum-forms}, the hypersurface $X \subset \PP(V)$ corresponding to the homogeneous form $F$ is smooth. Furthermore, a straightforward computation shows that $F$ is invariant by the diagonal automorphism $\tvarphi$ of order $p^r$ in $\GL(V)$ with signature
    $$\sigma=\big(\underbrace{1,(1-d),(1-d)^{2},\ldots,(1-d)^{\ell-1}}_{\ell\mbox{ times}},\underbrace{0,\ldots,0}_{n+2-\ell\mbox{ times}}  \big)\,.$$ 
    In the case where $p$ divides $d$, again the signature contains $0$ and $1$ which ensures that $\tvarphi$ has indeed order $p^r$. In the case where  $p$ does not divide $d(d-1)$ there may be no $0$ in the signature. Assume that $\tvarphi^{p^k}=\lambda\Id$ for some $k<r$. Then the first two coordinates in the signature yield $p^k\equiv p^k-dp^k\mod p^r$. This is only possible if $p^{r-k}$ divides $d$, a contradiction since in this case $p$ does not divide $d(d-1)$. Hence, in both cases, $\tvarphi$ induces an $F$-liftable $\varphi$ of $X$ of order $p^r$. This concludes the proof.
  \end{proof}

\begin{remark}
The statement of our \cite[Theorem~1.3]{GL13} is incorrectly stated to work for all $q$ relatively prime to $d$ and $d-1$ instead of $q=p^r$ a power of a prime number. This mistake was hinted in \cite[Theorem~5.1]{OY19} where our result is stated under the right hypothesis and our proof is reproduced word by word. Our proof above of Theorem~\ref{power-prime-order} is an adaptation of the proof of \cite[Theorem~1.3]{GL13}. The requirement for $q=p^r$ to be a power of a prime number comes from the fact that $\sigma_{k_0}$ is required to be relatively prime to $q=p^r$. This is not necessarily possible if $q$ is not a power of a prime number.
\end{remark}

  For fixed $n\geq 1$ and $d\geq 3$ with $(n,d)\neq (1,3), (2,4)$, Theorem~\ref{power-prime-order} can be easily used to compute the powers of a prime number that appear as the order of an $F$-liftable automorphism of a hypersurface of dimension $n$ and degree $d$ by computing the prime factorization of $(1-d)^\ell-1$ for $\ell\in\{1,2,\ldots,n+2\}$. Below we give examples of this technique for low-dimensional smooth cubic hypersurfaces. We begin by specializing Theorem~\ref{power-prime-order} to this setting.

\begin{corollary} \label{cor:cubics}
 Let $n,p,r$ be  positive integers with $n\geq 2$ and $p$ prime. Then $p^r$ is the order of an $F$-liftable automorphism of a smooth cubic hypersurface of dimension $n$ if and only if
  \begin{enumerate}[$(i)$]
  \item $p=2$ and $r\leq n+1$, or
  \item $p=3$ and there exists $\ell\in\{1,2,\ldots,n+1\}$ such that $(-2)^\ell\equiv 1 \mod p^r$.
    \item $p\neq 2,3$ and there exists $\ell\in\{1,2,\ldots,n+2\}$ such that $(-2)^\ell\equiv 1 \mod p^r$.
  \end{enumerate}
\end{corollary}

\begin{example}
  Let $n$ be a positive integer with $n\geq 2$ and $p$ prime. We will compute the values of $r$ such that  $p^r$ such that is the order of an $F$-liftable automorphism of a smooth cubic hypersurface of dimension $n$. Since $F$-liftable automorphisms of order $2^r$ are covered directly by Corollary~\ref{cor:cubics}~$(i)$, we are left to study the prime factorization of $(-2)^\ell - 1$ or, equivalently, the prime factorization of $2^\ell + (-1)^{\ell+1}$. For small values of $\ell$ we get the following factorizations.
 \begin{eqnarray*}
    2^1+(-1)^2=3 \qquad & 2^2+(-1)^3=3\qquad & 2^3+(-1)^4=3^2  \\
    2^4+(-1)^5=3\cdot 5 \qquad &2^5+(-1)^6=3\cdot 11 \qquad & 2^6 + (-1)^7 = 3^2\cdot 7 \\
    2^7+(-1)^8 = 3\cdot 43 \qquad & 2^8 + (-1)^9 = 3\cdot 5 \cdot 17 \qquad & 2^9+(-1)^{10} = 3^3\cdot 19.
  \end{eqnarray*}
We refer the interested reader to \cite{BLS75} for prime factorizations of $2^m \pm 1$ for large $m$. Now, let us take a closer look at low-dimensional examples.
\begin{itemize}
    \item[($n=2$)] By Proposition~\ref{proposition:char of liftability}, every automorphism of a smooth cubic surface is $F$-liftable. Therefore, Theorem~\ref{power-prime-order} yields that the powers of a prime number that appear as the order of an automorphism of a smooth cubic surface are $2^{r_2},3^{r_3}$ and $5$, where $r_2\leq 3$ and $r_3\leq 2$. This can be confirmed in \cite[Table~9.5]{Dol12}, where all cyclic groups acting on some smooth cubic surface are listed.
    
    \item[($n=3$)]  By Proposition~\ref{proposition:char of liftability}, every automorphism of a smooth cubic threefold is $F$-liftable. Therefore, Theorem~\ref{power-prime-order} yields that the powers of a prime number that appear as the order of an automorphism of a smooth cubic threefold are $2^{r_2},3^{r_3},5$ and $11$, where $r_2\leq 4$ and $r_3\leq 2$. This can be confirmed in \cite[Table~2]{WY19} where all the abelian groups acting on some cubic threefold are listed. See also their Proposition~5.1 that provides a similar computation using our earlier result \cite[Theorem~1.3]{GL13} that excluded the prime numbers $2$ and $3$.

    \item[($n=4$)] By Proposition~\ref{proposition:char of liftability}, we know that there are smooth cubic fourfolds admitting automorphisms that are not $F$-liftable, see Example~\ref{Klein-non-liftable}). On the other hand, Theorem~\ref{power-prime-order} yields that the powers of a prime number that appear as the order of an $F$-liftable automorphism of a smooth cubic fourfold are $2^{r_2},3^{r_3},5,7$ and $11$, where $r_2\leq 5$ and $r_3\leq 2$. It is worth noting that, after the work of Beauville and Donagi \cite{BD85}, we can attach to every smooth cubic fourfold $X$ an hyperk\"{a}hler fourfold given by its Fano variety of lines $F(X)$, and moreover $\Aut(X)$ is naturally isomorphic to the group of (polarized) automorphisms of $F(X)$ (see e.g. \cite[Proposition 4]{Cha12} or \cite[Corollary 2.3]{Fu16}). Since the hyperk\"{a}hler fourfold $F(X)$ carries a symplectic form $\omega$ it is natural to look at symplectic automorphisms of $X$, i.e. $\varphi\in \Aut(X)$ such that the corresponding $\widehat{\varphi}\in \Aut(F(X))$ satisfies $\widehat{\varphi}^*\omega = \omega$. It is proven in \cite{Fu16} that the powers of a primer number that appear as the order of a symplectic automorphism of a smooth cubic fourfold are $2^{s_2},3^{s_3},5,7$ and $11$, where $s_2\leq 3$ and $s_3\leq 2$ (see also \cite{Mon13} and \cite{LZ19}). The reader should note that Theorem~\ref{power-prime-order} is coherent with the latter result, since we have that every automorphism of order $5,7$ or $11$ is $F$-liftable (see Lemma \ref{non-liftable-single}), and in that case the symplectic condition stated in \cite[Lemma 3.2]{Fu16} follows from Lemma~\ref{determinant-3k} below. 
    
    \item[($n=5$)]  By Proposition~\ref{proposition:char of liftability}, every automorphism of a smooth cubic fivefold is liftable. Therefore, Theorem~\ref{power-prime-order} yields that the powers of a prime number that appear as the order of an automorphism of a smooth cubic fivefold are $2^{r_2},3^{r_3},5,7,11$ and $43$, where $r_2\leq 6$ and $r_3\leq 2$. 
\end{itemize}
\end{example}

 \begin{example}
   By Proposition~\ref{proposition:char of liftability}, every automorphism of a quartic threefold is liftable.  Theorem~\ref{power-prime-order} applied to this case yields that the powers of a prime number that appear as the order of an automorphism of a quartic threefold are $2^{r_2}$, $3^{r_3}$, $5$, $7$ and $61$ with $r_2,r_3\leq 4$. It is worth mentioning the remarkable fact that for a smooth quartic threefold $X$ the automorphism group $\Aut(X)$ and group of birational automorphisms $\operatorname{Bir}(X)$ coincide thanks to the seminal work of Iskovskikh and Manin \cite{IM71}, which is considered the starting point of birational rigidity (see e.g. \cite{Ko19b} for an historical account and new results).
 \end{example}

 For the remaining of this section, we fix a prime number $p$ not dividing $d-1$ and we will make use of the equivalence relation in the $\ZZ/p^r\ZZ$ given by $a\sim b$ if and only of $b\equiv (1-d)^ia\mod p^r$ for some $i\in \ZZ$. The relation $\sim$ is indeed an equivalence relation since $1-d$ is invertible in $\ZZ/p^r\ZZ$. Let now $\tvarphi$ be an automorphism of order $p^r$ in $\GL(V)$. By Lemma~\ref{diagonalization}, all the eigenvalues of $\tvarphi$ are $p^r$-roots of unity. Let $\xi$ be a primitive $p^r$-root of unity. For every $a\in \ZZ/p^r\ZZ$, we denote by $V(a)$ the eigenspace of $V$ associated to the eigenvalue $\xi^{a}$.

\begin{lemma}\label{equal-dim}
  Let $X$ be a smooth hypersurface of dimension $n\geq1$ and degree $d\geq3$ in $\PP(V)$  with $(n,d)\neq (1,3), (2,4)$. Let $\varphi$ be an $F$-liftable automorphism of order $p^r$ with $p$ prime not dividing $d-1$ and let $\tvarphi$ be an $F$-lifting of $\varphi$. If $a\sim b$, then $\dim V(a)=\dim V(b)$. 
\end{lemma}

\begin{proof}
  If $p^r$ divides $d$, then $1-d\equiv 1\mod p^r$. Hence $a\sim b$ if and only if $a\equiv b\mod p^r$ and so $V(a)=V(b)$. In this case the lemma holds trivially. Assume now $p^r$ does not divide $d$.  Letting $i\in \ZZ$ and $a\in \ZZ/p^r\ZZ$, we let $V_0=V((1-d)^ia)$ and $V_1=V((1-d)^{i+1}a)$. We have $(1-d)^i\not\equiv(1-d)^{i+1}\mod p^r$ and so $V_0$ and $V_1$ are eigenspaces associated to different eigenvalues. Hence, $V_0\cap V_1=\{0\}$. Let $V_2$ be a complementary of $V_0\oplus V_1$ in $V$. Since $\varphi^*(F)=F$ we have that
  $$F\in S^{d-1}(V_0^*)\otimes S^1(V_1^*)\oplus S^1(V_0^*)\otimes S^{d-1}(V_1^*\oplus V_2^*)\oplus S^d(V_1^*\oplus V_2^*)\,.$$ %
  Since $X$ is smooth, by Lemma~\ref{simple}, we conclude that $\dim V((1-d)^ia)\leq \dim V((1-d)^{i+1}a)$. Furthermore, let $\ell$ be any integer such that $(1-d)^\ell\equiv 1\mod p^r$ and let $k$ be such that $k\ell\leq i\leq (k+1)\ell$. We conclude
  $$\dim V(a)=\dim V((1-d)^{k\ell}a)\leq \dim V((1-d)^ia)\leq \dim V((1-d)^{(k+1)\ell}a)=\dim V(a)\,. $$
  This yields $V((1-d)^{i} a)=V(a)$ for every $i$ and so the lemma follows.
\end{proof}

For every $a\in \ZZ/p^r\ZZ$, we let $\overline{a}$ be the class of $a$ under $\sim$. Furthermore, if $a\not \equiv 0\mod p^r$ we can write $a=p^ka'$ with $\gcd(p,a')=1$ and $k<r$. We have $a\equiv (1-d)^ia \mod p^r$ if and only if $1\equiv (1-d)^i \mod p^{r-k}$. Hence, letting $\ell$ be the smallest positive integer such that $(1-d)^\ell\equiv 1\mod p^{r-k}$, we obtain that
$$\overline{a}=\left\{a, (1-d)a,(1-d)^2a,\ldots,(1-d)^{\ell-1}a\right\}\subset \ZZ/p^r\ZZ$$
and the cardinality of $\overline{a}$ is exactly $\ell$, i.e., all the elements listed above are different. 

\begin{lemma} \label{sum-multiple} 
 Let $p$ be a prime number not dividing $d-1$ and let $d=p^sd'$ with $\gcd(p,d')=1$.  then for every $a\in \ZZ/p^r\ZZ$ we have that the sum $\sum \overline{a}$ of all elements in the class $\overline{a}$ is congruent to $0$ modulo $p^{r-s}$.
\end{lemma} 

\begin{proof}
  If $a\equiv 0\mod p^r$ then the lemma holds trivially. Assume $a\not \equiv 0\mod p^r$ we can write $a=p^ka'$ with $\gcd(p,a')=1$ and $k<r$. 
  Letting $\ell$ be the smallest positive integer such that $(1-d)^\ell\equiv 1 \mod p^{r-k}$ we let $c\in \ZZ$ be such that $(1-d)^\ell=1+cp^{r-k}$. We then have 
  $$\sum \overline{a}=\sum_{i=0}^{\ell-1}(1-d)^ia=a\cdot\frac{1-(1-d)^\ell}{1-(1-d)}=-p^ka'\cdot\frac{cp^{r-k}}{p^sd'}=-\frac{1}{d'}a'cp^{r-s}$$

The number in the above equation is an integer by the definition of the sum. Hence, we have that $d'$ divides $a'c$ and so the sum is a multiple of $p^{r-s}$. 
\end{proof}

The application of the above results that will be needed in the sequel is the following lemma.

\begin{lemma} \label{determinant-3k} %
  Let $X$ be a smooth hypersurface of dimension $n\geq 1$ and degree $d\geq 2$ in $\PP(V)$ with $(n,d)\neq (1,3), (2,4)$ given by a homogeneous form $F\in S^d(V^*)$. Let $\varphi$ be an automorphism of order $p^r$ and let $\tvarphi$ be an $F$-lifting, with $p$ prime. If $p$ does not divide $d-1$, then the determinant of $\tvarphi$ is a $p^{\min(s,r)}$-root of unity, where $d=p^sd'$ with $\gcd(p,d')=1$. In particular, if $p$ does not divide $d$, then the determinant of $\tvarphi$ is one.
\end{lemma}
 
\begin{proof}
  By Lemma~\ref{diagonalization}, the automorphism $\tvarphi$ is diagonalizable and so its determinant is given by $\det(\tvarphi)=\prod\lambda_i^{n_i}$, where $\{\lambda_i\}$ is the set of eigenvalues and $n_i$ is the dimension of the eigenspace asociated to the eigenvalue $\lambda_i$. If $r\leq s$ the lemma follows trivially since by Lemma~\ref{diagonalization} all the eigenvalues of $\tvarphi$ are $p^r$-roots of unity. In the sequel we assume $r>s$. Letting $\xi$ be a primitive $p^r$-root of unity, the determinant of $\tvarphi$ is given by
  $$\det(\tvarphi)=\prod_{a\in\ZZ/p^r\ZZ}(\xi^{a})^{\dim V(a)}$$
  Letting $C$ be the set of equivalence classes by the relation $\sim$, we can write this last equation as
    $$\det(\tvarphi)=\prod_{\overline{b}\in C} \prod_{a\in \overline{b}}(\xi^{a})^{\dim V(a)}$$  
  By Lemma~\ref{equal-dim}, the dimension of $V(a)$  is constant when $a$ varies in the equivalence class $\overline{b}$ so we have
  $$\det(\tvarphi)=\prod_{\overline{b}\in C} \left(\xi^{\sum \overline{b}}\right)^{\dim V(b)}$$ Finally, the result follows from Lemma~\ref{sum-multiple} since  $\sum \overline{b}\equiv 0\mod p^{r-s}$ implies that $\xi^{\sum \overline{b}}$ is a $p^s$-root of unity.
\end{proof}

\section{Liftability of the  automorphism group}
\label{sec:liftability-aut}

In this section we will prove our main theorem stating that $\Aut(X)$ is $F$-liftable for every smooth hypersurface of dimension $n$ and degree $d$ if and only if $d$ and $n+2$ are relatively prime.

\begin{lemma}\label{lemma:SL}
  Let $X$ be a smooth hypersurface of dimension $n\geq 1$ and degree $d\geq 3$ in $\PP(V)$  with $(n,d)\neq (1,3), (2,4)$ and let $\varphi$ be an automorphism of $X$ of order $p^r$ with $p$ a prime number dividing $d$ and $r\geq 1$. If $p$ does not divide $n+2$, then $\varphi$ admits an unique $F$-lifting to $\SL(V)$.
\end{lemma}

\begin{proof}
  Let $d=p^sd'$ with $\gcd(p,d')=1$. Since $p$ does not divide $n+2$, there exists an $F$-lifting $\tvarphi$ by Proposition~\ref{proposition:char of liftability}. Furthermore, by Lemma~\ref{determinant-3k}~$(i)$ we have $\det(\tvarphi)=\xi^c$ where $\xi$ is a primitive $p^{\min(s,r)}$-root of unity. Since $p$ does not divide $n+2$, by B\'ezout identity, there exists $a$ and $b$ such that $p^{\min(s,r)}a-(n+2)b=c$. The automorphism $\xi^b\tvarphi$ belongs $\SL(V)$ since $\det(\xi^b\tvarphi)=\xi^{(n+2)b}\xi^c=1$. Furthermore, $\xi^b\tvarphi$ is again an $F$-lifting of $\varphi$ by Lemma~\ref{basic-liftings} since $\xi$ is a $p^{\min(s,r)}$-root of unity and $p^{\min(s,r)}$ equals $\gcd(d,p^r)$. 
  
  Finally, let $\tvarphi$ and $\tvarphi'$ be two $F$-liftings of $\varphi$ in $\SL(V)$. By Lemma~\ref{basic-liftings}, we have $\tvarphi'=\lambda \tvarphi$ with $\lambda$ a $p^{\min(s,r)}$-root of unity. Moreover, we have  $$1=\det(\tvarphi)=\lambda^{n+2}\det(\tvarphi')=\lambda^{n+2}$$
Hence, $\lambda$ is also a $(n+2)$-root of unity. Since $p$ does not divide $n+2$, we conclude $\lambda=1$.
\end{proof}

We will apply the above lemma to prove that whenever $p$ divides $d$ and does not divide $n+2$, then every $p$-group in $\Aut(X)$ admits a unique $F$-lifting to $\SL(V)$.

\begin{lemma} \label{lemma:3-group} %
  Let $X$ be a smooth hypersurface of dimension $n\geq 1$ and degree $d\geq 3$ in $\PP(V)$  with with $\gcd(d,n+2)=1$ and let $G\subseteq \Aut(X)$ be a $p$-subgroup for some prime number $p$. If $p$ divides $d$, then $G$ admits a unique $F$-lifting contained in $\SL(V)$.
\end{lemma}

\begin{proof}
  Since $\gcd(d,n+2)=1$ and $p$ divides $d$, we have $p$ does not divide $n+2$. By Lemma~\ref{lemma:SL}, every element $\varphi$ in $G$ admits an unique $F$-lifting $\tvarphi$ to $\SL(V)$.  Let $\tG$ be the group generated in $\SL(V)$ by $\widetilde{H}$, where
  $$\widetilde{H}=\{\tvarphi\in \SL(V)\mid \tvarphi  \mbox{ the unique $F$-lifting of $\varphi\in G$ to }\SL(V)\}\,.$$
  The restriction of the canonical projection $\pi\colon\GL(V) \to \PGL(V)$ to $\widetilde{G}$ gives a surjective morphism $\pi|_{\tG}:\tG \to G$ and $\tvarphi^*(F)=F$ for every $\tvarphi\in \tG$. Let $\tvarphi$ be in the kernel of $\pi|_{\tG}$, then $\tvarphi=\lambda\Id$ and since $\tvarphi^*(F)=F$ we obtain $\lambda^d=1$. Furthermore, $\tvarphi\in \SL(V)$ so $1=\det(\tvarphi)=\lambda^{n+2}$. This yields $\lambda=1$ since $\gcd(d,n+2)=1$ and so $\pi|_{\tG}$ is an isomorphism.

    In particular, we proved that $\tG=\widetilde{H}$ and so the uniqueness statement follows from the uniqueness of the $F$-lifting of every element in $G$ to $\SL(V)$.
\end{proof}

The following lemma follows directly from  \cite[Theorem~4.8]{OY19} in the case where $d$ is prime. This result is proven applying the Hochschild-Serre exact sequence in that case. We remark that, by \cite[Theorem~4.8]{OY19}, the condition that that $\gcd(d,n+2)=1$ in the lemma could be removed when $d$ is prime. In the case where $\gcd(d,n+2)=1$, we are able to give a low-tech proof where we replace the role played by the Hochschild-Serre exact sequence in the proof of \cite[Theorem~4.8]{OY19} with the determinant map.

\begin{lemma} \label{lemma:other-p} %
  Let $X$ be a smooth hypersurface of dimension $n\geq 1$ and degree $d\geq 3$ in $\PP(V)$  with with $\gcd(d,n+2)=1$ and let $G\subseteq \Aut(X)$ be a subgroup. If $d$ is also relatively prime to the order of $G$, then $G$ admits a unique $F$-lifting.
\end{lemma}

\begin{proof}
  By Proposition~\ref{proposition:char of liftability}, every element $\varphi$ in $G$ admits an $F$-lifting $\tvarphi$, Morever, since $d$ and $|G|$ are relatively prime, the same follows for $d$ and the order of $\varphi$ and so by Lemma~\ref{basic-liftings}~$(i)$ the $F$-lifting $\tvarphi$ of $\varphi$ is unique. Let $\tG$ be the group generated by $\widetilde{H}$, where
  $$\widetilde{H}=\{\tvarphi\in \GL(V)\mid \tvarphi  \mbox{ the unique $F$-lifting of $\varphi\in G$ to }\GL(V)\}\,.$$
  
  The restriction of the canonical projection $\pi\colon\GL(V) \to \PGL(V)$ to $\widetilde{G}$ gives a surjective morphism $\pi|_{\tG}:\tG \to G$ and $\tvarphi^*(F)=F$ for every $\tvarphi\in \tG$. Let $\tvarphi$ be in the kernel of $\pi|_{\tG}$, then $\tvarphi=\lambda\Id$ and since $\tvarphi^*(F)=F$ we obtain $\lambda^d=1$. In particular, the determinant $\det(\tvarphi)=\lambda^{n+2}$ is a $d$-root of unity by Lemma~\ref{diagonalization}.
  
  Let now $\tvarphi=\tvarphi_1\circ\tvarphi_2\circ\ldots\circ\tvarphi_\ell$ where each $\tvarphi_i\in \widetilde{H}$. Taking determinant we obtain $$\det(\tvarphi)=\det(\tvarphi_1)\cdot\det(\tvarphi_2)\cdots\det(\tvarphi_\ell)$$ Again by Lemma~\ref{diagonalization} we have that $\det(\tvarphi_i)$ is a $q_i$-root of unity, where $q_i$ is the order of $\tvarphi_i$. In particular, $\det(\tvarphi)$ is a $q_1\cdot q_2\cdots q_\ell$-root of unity. Since $d$ and   $q_1\cdot q_2\cdots q_\ell$ are relatively prime, we obtain that $\det(\tvarphi)=1$. Finally, $\lambda^{n+2}=\det(\tvarphi)=1$ implies that $\lambda=1$ since $d$ and $n+2$ are relatively prime. This yields that $\pi|_{\tG}$ is an isomorphism.

  In particular, we proved that $\tG=\widetilde{H}$ and so the uniqueness statement follows from the uniqueness of the $F$-lifting of every element in $G$.
\end{proof}

\begin{proposition} \label{most-p-SL} %
  Let $X$ be a smooth hypersurface of dimension $n\geq 1$ and degree $d\geq 3$ in $\PP(V)$ with with $\gcd(d,n+2)=1$ and let $G_p$ be a Sylow $p$-subgroup of $\Aut(X)$ with $p$ a prime number. Then the following hold.
  \begin{enumerate}[$(i)$]
  \item If $p$ does not divide $d(d-1)$, then there exists a unique $F$-lifting of $G_p$ and this $F$-lifting is contained in $\SL(V)$.
  \item If $p$ divides $d-1$ then there exists a unique $F$-lifting of $G_p$.
  \item If $p$ divides $d$, then there exists a unique $F$-lifting of $G_p$ that is contained in $\SL(V)$.
  \end{enumerate}
\end{proposition}
\begin{proof}
  Lemma~\ref{lemma:other-p} imply the existence of a unique $F$-lifting $\tG_p$ to $\GL(V)$ of $G_p$ in cases $(i)$ and $(ii)$. Lemma~\ref{determinant-3k} imply that in case $(i)$ the $F$-lifting $\tG_p$ is contained in $\SL(V)$. Finally, part $(iii)$ follows directly from Lemma~\ref{lemma:3-group}.
\end{proof}

The following is our main result in this paper. Our original idea for the proof was inspired by \cite[Theorem 4.11]{WY19}.

\begin{theorem} \label{all-liftable} Let $n\geq 1$ and $d\geq 3$ with $(n,d)\neq (1,3), (2,4)$. Then the automorphism group of every smooth hypersurface of dimension $n$ and degree $d$ in $\PP(V)$ is $F$-liftable if and only if $d$ and $n+2$ are relatively prime.
\end{theorem}

\begin{proof}
  As in the case of Proposition~\ref{proposition:char of liftability}, the ``only if'' part follow directly from Example~\ref{Klein-non-liftable} that provides automorphisms that are not $F$-liftable of every smooth Klein hypersurface of dimension $n$ and degree $d$ whenever $d$ and $n+2$ are not relatively prime.
  
  To prove the ``if'' part, let $X\subseteq \PP(V)$ be a smooth hypersurface of dimension $n$ and degree $d$ given by the homogeneous form $F\in S^d(V^*)$.  Assume that $d$ and $n+2$ are relatively prime. By Theorem~\ref{matsumura-monsk}, we have $\Aut(X)=\Lin(X)$ and is finite. Let us write
  $$|\Aut(X)|=p_1^{k_1}\cdots p_r^{k_r}q_1^{l_i}\cdots q_s^{l_s}\,,$$
  where $p_1,\ldots,p_r,q_1\ldots q_s$  are pairwise distinct primes,  $k_1,\ldots,k_r,l_1,\ldots,l_s\geq 1$ are integers and $q_1,\ldots q_s$ are all the prime numbers dividing $d-1$. By Proposition~\ref{most-p-SL}~$(i)$ and $(iii)$, every $p_i$-Sylow subgroup $G_{p_i}$ of $\Aut(X)$ admits an unique $F$-lifting $\tG_{p_i}\subset \SL(V)$, for all $i=1,\ldots,r$. By Proposition~\ref{most-p-SL}~$(ii)$, every $q_i$-Sylow subgroup $G_{q_i}$ of $\Aut(X)$ admits an unique $F$-lifting $\tG_{q_i}\subset \GL(V)$, for all $i=1,\ldots,s$.  We will check that the group $\tG$ generated by 
  $$\tG_{p_1}\cup\ldots \cup \tG_{p_r}\cup \tG_{q_1}\cup\ldots \cup \tG_{q_s}$$
  is a lifting of $\Aut(X)$. The restriction of the canonical projection $\pi\colon\GL(V) \to \PGL(V)$ to $\widetilde{G}$ gives a surjective morphism $\pi|_{\tG}:\tG \to \Aut(X)$.  If $\tvarphi$ is in the kernel of $\pi|_{\tG}$, then $\tvarphi=\lambda \Id$, for some $\lambda \in \CC^\times$. Furthermore, since $\tvarphi^*(F)=F$ we obtain that $\lambda^d=1$ and so $\lambda$ is a $d$-root of unity. In particular, $\det(\tvarphi)=\lambda^{n+2}$ is also a $d$-root of unity.

  On the other hand, $\tvarphi=\tpsi_1\cdots\tpsi_\ell$ where each $\tpsi_i$ belongs to some $\tG_p$, with $p\in\{p_1,\ldots,p_r,q_1,\ldots,q_s\}$. Furthermore, if $\tpsi_i\in \tG_{p_j}$ then $\det(\tpsi_i)=1$ since in this case $\tpsi_i\in \SL(V)$. Furthermore, if $\tpsi_i\in \tG_{q_j}$ then the order of $\tpsi_i$ divides $q_j^{l_j}$ and so $\det(\tpsi_i)$ is a $q_j^{l_j}$-root of unity by Lemma~\ref{diagonalization}. In particular, we have shown that $\det(\psi_i)$ is a $q_1^{l_i}\cdots q_s^{l_s}$-root of unity for every $i=1,\ldots, \ell$. Hence,
  $$\det(\tvarphi)=\det(\psi_1)\cdots\det(\psi_\ell)\mbox{ is a $q_1^{l_i}\cdots q_s^{l_s}$-root of unity}\,.$$
  We have shown that $\det(\tvarphi)=\lambda^{n+2}$ is a $q_1^{l_i}\cdots q_s^{l_s}$-root of unity and a $d$-root of unity. Since $q_j$ are the prime numbers dividing $d-1$, they do not divide $d$ and so $\gcd(d,q_1^{l_i}\cdots q_s^{l_s})=1$. Hence $\lambda^{n+2}=1$. This last equality gives that $\lambda$ is a $(n+2)$-root of unity and we showed before that $\lambda$ is also a $d$-root of unity. Hence $\lambda=1$ since $\gcd(d,n+2)=1$. This yields that $\pi|_{\tG}$ is also injective, concluding the proof.
\end{proof}

\section{A bound for the order of certain $p$-groups acting on smooth hypersurfaces}
\label{sec:applications}

As an application of our above results we show the following proposition giving a sufficient condition for the order of $\Aut(X)$ to not be divisible by $p^2$, where $X$ is a smooth hypersurfaces of dimension $n$ and degree $d$. If $p$ does not divide $d(d-1)$, we define $\ell(p^r)$ as the smallest positive integer 
such that $(1-d)^{\ell(p^r)}\equiv 1 \mod p^r$. Such an $\ell(p^r)$ always exists since $d-1$ and $p$ are relatively prime. 

\begin{proposition} \label{sylow-small}
Let $X$ be a smooth hypersurface of dimension $n\geq 1$ and degree $d\geq 3$ in $\PP(V)$ with with $\gcd(d,n+2)=1$ and let $G_p$ be a Sylow $p$-subgroup of $\Aut(X)$ with $p$ a prime number not dividing $d(d-1)$. Assume that $G_p$ is $F$-liftable. If $\ell(p^2)> n+2$ and $2\ell(p)>n+2$, then $G_p$ is trivial or isomorphic to $\ZZ/p\ZZ$.
\end{proposition}

\begin{proof}
  Assume by contradiction that the automorphism group $\Aut(X)$ of a smooth hypersurface $X$ of dimension $n$ and degree $d$ given by the homogeneous form $F\in S^d(V^*)$ has a Sylow $p$-subgroup $G_p$ of order $p^r$ with $r>1$. This, in particular, yields that $\Aut(X)$ has a subgroup $H$ of order $p^2$ (see e.g. \cite[Ch. 2, Theorem 1.9]{Suz82}). A standard fact about $p$-groups is that, up to isomorphism, $\ZZ/p^2\ZZ$ and $\ZZ/p\ZZ\times \ZZ/p\ZZ$ are the only two groups of order $p^2$ (see e.g. \cite[Example of Ch. 1, \S 3]{Suz82}). By Theorem~\ref{power-prime-order}, the condition $\ell(p^2)>n+2$ implies that $H\not\cong \ZZ/p^2\ZZ$. Hence we conclude that $H\cong \ZZ/p\ZZ\times \ZZ/p\ZZ$.

The group $H$ is $F$-liftable by Proposition~\ref{most-p-SL}~$(i)$. Let $\tvarphi$ and $\tvarphi'$ be $F$-liftings of the generator $\varphi$ and $\varphi'$ of each $\ZZ/p\ZZ$ factor inside $H$, respectively. Since $\tvarphi$ and $\tvarphi'$ commute, there exists a basis $\beta$ of $V$ such that the matrices of both automorphisms are diagonal (see e.g. \cite[Ch. VII, \S 5, $n^{\circ}$ 7, Prop. 13]{Bou81}). 

Fix a primitive $p$-root of unity $\xi$ and let $\sigma$ and $\sigma'$ be the signatures of $\tvarphi$ and $\tvarphi'$, respectively. Without loss of generality, we may assume that the first $\ell(p)$ coordinates of $\sigma$ are $1,(1-d),\ldots,(1-d)^{\ell(p)-1}$. This yields that $x_i^d$ does not appear in $F$ with non-zero coefficient for $i<\ell(p)$. Hence, Lemma~\ref{simple} implies that $x^{d-1}_ix_{i+1}$ and $x^{d-1}_{\ell(p)-1}x_{0}$ must appear in $F$ with non-zero coefficient for $i<\ell(p)$.

Now, Lemma~\ref{equal-dim} yields $\dim V((1-d)^i)=c$ is constant for all $i<\ell(p)$, where $V(a)$ is the dimension of the eigenspace of $V$ associated to the eigenvalue $\xi^a$ of $\varphi$. And, since $2\ell(p)>n+2=\dim V$ we conclude that $c=1$. Furthermore, a similar argument shows that the dimension of every eigenspace $V(a)$ different from $V((1-d)^i)$ for $i<\ell(p)$ must be zero except for $V(0)$. This yields that 
$$\sigma=\big(\underbrace{1,(1-d),(1-d)^{2},\ldots,(1-d)^{\ell(p)-1}}_{\ell(p)\mbox{ times}},\underbrace{0,\ldots,0}_{n+2-\ell(p)\mbox{ times}}  \big)\,.$$ 

Assume now that $\sigma'_i$ the $i$-th coordinate of $\sigma'$ is $a\not\equiv 0\mod p$ with $i<\ell(p)$.  Then a similar argument as above shows that

\begin{align*}
\sigma'&=\big(\underbrace{a(1-d)^{\ell(p)-i},\ldots,a(1-d)^{\ell(p)-1}}_{\ell(p)-i\mbox{ times}},\underbrace{a,a(1-d),\ldots,a(1-d)^{\ell(p)-i-1}}_{\ell(p)-i\mbox{ times}},\underbrace{0,\ldots,0}_{n+2-\ell(p)\mbox{ times}}  \big) \\
  & =\big(\underbrace{b,b(1-d),b(1-d)^{2},\ldots,b(1-d)^{\ell(p)-1}}_{\ell(p)\mbox{ times}},\underbrace{0,\ldots,0}_{n+2-\ell(p)\mbox{ times}}  \big)\,,
\end{align*}
where $b=a(1-d)^{\ell(p)-i}$. This yields a contradiction, since in this case, $\varphi'=\varphi^b$ and so they belong to the same cyclic group. Hence, $\sigma'_i=0$ for all $i<\ell(p)$.

  We have proven that the signature $\sigma'$ must have the form
  $$\sigma'=\big(\underbrace{0,\ldots,0}_{\ell(p)\mbox{ times}},\underbrace{1,(1-d),(1-d)^{2},\ldots,(1-d)^{\ell(p)-1}}_{\ell(p)\mbox{ times}},\underbrace{*,\ldots,*}_{n+2-2\ell(p)\mbox{ times}}  \big)\,.$$ 
But this is impossible since $2\ell(p)>n+2$. This contradiction implies that $G_p$ has order at most $p$ proving the proposition.
\end{proof}

We have the following corollary that is the announced result for this section.

\begin{corollary} \label{sylow-small-order}
Let $X$ be a smooth hypersurface of dimension $n\geq 1$ and degree $d\geq 3$ in $\PP(V)$ with $\gcd(d,n+2)=1$. Let $p$ be a prime number not dividing $d(d-1)$. If $\ell(p^2)> n+2$ and $2\ell(p)>n+2$, then $p^2$ does not divide the order of $\Aut(X)$.
\end{corollary}

\begin{proof}
By Theorem~\ref{all-liftable} we have that $\Aut(X)$ is $F$-liftable and so, in particular a $p$-Sylow subgroup of $\Aut(X)$ is also $F$-liftable. The corollary now follows directly from Proposition~\ref{sylow-small}.
\end{proof}

\begin{example} \label{orders-full-cubics}
Corollary~\ref{sylow-small-order} is particularly useful for small values of $n$. For instance, when $X$ is a smooth cubic hypersurface of dimension $n$ and $3$ does not divide $n+2$, we obtain the following bounds by directly applying Corollary \ref{sylow-small-order}.
\begin{itemize}
    \item[($n=2$)] The order of $\Aut(X)$ is  $2^{r_2}3^{r_3}5^{r_5}$, where $r_5\leq 1$.
    \item[($n=3$)]  The order of $\Aut(X)$ is  $2^{r_2}3^{r_3}5^{r_5}11^{r_{11}}$, where $r_5,r_{11}\leq 1$.
    \item[($n=4$)] The order of $\Aut(X)$ is  $2^{r_2}3^{r_3}5^{r_5}7^{r_7}11^{r_{11}}$, where $r_5,r_7,r_{11}\leq 1$. This case does not follow directly from Corollary~\ref{sylow-small-order}. Nevertheless, the bounds for $r_5,r_7,r_{11}$ follow from the corresponding bounds in the case $n=5$.
    \item[($n=5$)] The order of $\Aut(X)$ is  $2^{r_2}3^{r_3}5^{r_5}7^{r_7}11^{r_{11}}43^{r_{43}}$, where $r_5,r_7,r_{11},r_{43}\leq 1$.
    \end{itemize}
\end{example}

\begin{example}
  In the case of a smooth quintic threefold $X$, it follows from \cite[Theorem 4.8]{OY19} that the $p$-Sylow subgroups of $\Aut(X)$ are $F$-liftable for $p\neq 5$. Proposition~\ref{sylow-small} yields that the order of $\Aut(X)$ is  $2^{r_2}3^{r_3}5^{r_5}13^{r_{13}}17^{r_{17}}41^{r_{41}}$, where $r_{13},r_{17},r_{41}\leq 1$. This retrieves part of \cite[Theorems~5.8, 5.9 and 5.10]{OY19} in a uniform way.
\end{example}

\bibliographystyle{alpha}
\bibliography{nfolds}
\end{document}